\documentclass[11pt]{amsart}

\usepackage{amsfonts}
\usepackage{amsmath,amssymb,amsthm}
\usepackage{enumerate,graphicx}
\usepackage[square,numbers]{natbib}
\usepackage{stmaryrd}
\usepackage{MnSymbol}
\usepackage[all]{xy}
\newtheorem{theorem}{Theorem}[section]

\newtheorem{claim}[theorem]{Claim}

\newtheorem{corollary}[theorem]{Corollary}
\newtheorem{lemma}[theorem]{Lemma}

\theoremstyle{definition}

\newtheorem{example}[theorem]{Example}

\newtheorem{problem}{Problem}
\newtheorem{remark}[theorem]{Remark}

\begin{document}

\title[Bernoulli Automorphism of Reversible LCA]{On the Bernoulli Automorphism of Reversible Linear Cellular Automata}


\author{Chih-Hung Chang}
\address{Department of Applied Mathematics, National University of Kaohsiung, Kaohsiung 81148, Taiwan, ROC.}
\email{chchang@nuk.edu.tw; hchang@nuk.edu.tw}

\author{Huilan Chang}

\thanks{This work is partially supported by the Ministry of Science and Technology, ROC (Contract No MOST 103-2115-M-390-002- and 103-2115-M-390-004-).}
\date{September 21, 2015}

\baselineskip=1.5\baselineskip

% -------------------------------------------------------------
\begin{abstract}
This investigation studies the ergodic properties of reversible linear cellular automata over $\mathbb{Z}_m$ for $m \in \mathbb{N}$. We show that a reversible linear cellular automaton is either a Bernoulli automorphism or non-ergodic. This gives an affirmative answer to an open problem proposed in [Pivato, Ergodic theory of cellular automata, Encyclopedia of Complexity and Systems Science, 2009, pp.~2980-3015] for the case of reversible linear cellular automata.
\end{abstract}

\maketitle

% -------------------------------------------------------------
\section{Introduction}\label{sec:intro}

Motivated by biological applications, John von Neumann introduced cellular automata (CAs) in the late $1940$s. The main goal was to design self-replicating artificial systems that are also computationally universal and are analogous to human brain. Namely, CA is designed as a computing device in which the memory and the processing units are not separated from each other, and is massively parallel and capable of repairing and building itself given the necessary raw material.

CA has been systematically studied by Hedlund from purely mathematical point of view \cite{Hed-MST1969}. For the past few decades, studying CA from the viewpoint of the ergodic theory has received remarkable attention \cite{ABC-JCA2013, BKM-PD1997, CDL-TCS2004, CFMM-TCS2000, HMM-DaCDS2003, Kle-PAMS1997, PY-ETDS2002, PY-ETDS2004}. Pivato has characterized the invariant measures of bipermutative right-sided, nearest neighbor cellular automata \cite{Pivato-DaCDS2005}. Moreover, Pivato and Yassawi introduced the concepts of harmonic mixing for measures and diffusion for a linear CA and developed broad sufficient conditions for convergence the limit measures \cite{PY-ETDS2002, PY-ETDS2004}. Sablik demonstrates the measure rigidity and directional dynamics for CA \cite{Sab-ETDS2007, Sab-TCS2008}. Host \emph{et al}. have studied the role of uniform Bernoulli measure in the dynamics of cellular automata of algebraic origin \cite{HMM-DaCDS2003}. Furthermore, the sufficient conditions whether a one-dimensional permutative CA is strong mixing, $k$-mixing, or Bernoulli automorphic were independently revealed by Kleveland and Shereshevsky \cite{Kle-PAMS1997, She-MM1992, She-IMN1997}. Recently, one-sided expansive invertible cellular automata and two-sided expansive permutation cellular automata have been demonstrated to be strong mixing (see \cite{BM-IJM1997, BM-JMSJ2000, JNY-AAM2013, JY-AAM2007, Nasu-TAMS2002}).

Almost all the results about  are for one-dimensional (mostly permutative) CA and for the uniform measure. It is natural to ask the following question:
\begin{problem}[See \cite{Pivato-2009}]
Can mixing and ergodicity be obtained for non-permutative CA and/or non-uniform measures? What about multidimensional CA?
\end{problem}
Theorem \ref{thm:ILCA-Bernoulli} and Corollary \ref{cor:ILCA-NotErgodic} indicate that an invertible linear CA is either Bernoulli automorphic or non-ergodic for the uniform Bernoulli measure. In \cite{CFMM-TCS2000}, Cattaneo \emph{et al.} address a necessary and sufficient condition for the ergodicity of linear CA. Corollary \ref{cor:ILCA-NotErgodic} reveals a concise condition for the ergodicity of invertible CA. The result remains true for those measures satisfying some conditions (see Remark \ref{rmk:Result-for-more-measure}). The methodology can be extended to the investigation of multidimensional invertible linear CA, and even possible to the non-permutative cases. Related works are under preparation.

The rest of this paper is organized as follows. Section \ref{sec:main-result} states the main results and some preliminaries. The proofs are postponed to Sections \ref{sec:Proof-mixing} and \ref{sec:Proof-Bernoulli} while the key ideas are revealed via some examples in Section \ref{sec:example}. Discussion and further works are addressed in Section \ref{sec:discuss}.

% -------------------------------------------------------------
\section{Statement of Main Results}\label{sec:main-result}

Let ${\mathbb{Z}}_{m}=\{0, 1,\ldots, m-1\}$ be the ring of the integers modulo $m$, where $m \geq 2$, and let $\mathbb{Z}^{\mathbb{Z}}_{m}$ be the space of all doubly-infinite sequences $x=(x_n)_{n=-\infty}^{\infty}\in \mathbb{Z}^{\mathbb{Z}}_{m}$, equipped with the product of the Tychonoff topology. Then the shift $\sigma: \mathbb{Z}^{\mathbb{Z}}_{m} \rightarrow \mathbb{Z}^{\mathbb{Z}}_{m}$ defined by $(\sigma x)_{i} = x_{i+1}$ is a homeomorphism of the compact metric space $\mathbb{Z}^{\mathbb{Z}}_{m}$. A one-dimension cellular automaton is a continuous map
$T_{f}:\mathbb{Z}^{\mathbb{Z}}_{m}\rightarrow \mathbb{Z}^{\mathbb{Z}}_{m}$ defined by $(T_{f} x)_i = f(x_{i+l},\ldots, x_{i+r})$, where $l, r \in \mathbb{Z}$ and $f: \mathbb{Z}^{r-l+1}_{m}\rightarrow \mathbb{Z}_{m}$ is a given local rule or map. A local rule $f$ is said to be linear if it can be written as
$$
f(x_l, \ldots, x_r) = \Sigma_{i=l}^r \lambda_i x_i \pmod{m}, \quad l, r \in \mathbb{Z},
$$
where at least one among $\lambda_{l},\ldots,\lambda _{r}$ is nonzero in $\mathbb{Z}_m$ \cite{FLM-TCS1997,MOW-CMP1984}.

A local rule $f$ is said to be permutative in $x_j$ (or, at the index $j$) if for any given finite
sequence
$$
(\overline{x}_l,\ldots,\overline{x}_{j-1},\overline{x}_{j+1},\ldots,\overline{x}_{r}) \in \mathbb{Z}_{m}^{r-l}
$$
we have
$$
\{f(\overline{x}_l,\ldots,\overline{x}_{j-1},x_{j}, \overline{x}_{j+1},\ldots,\overline{x}_{r}):x_{j} \in \mathbb{Z}_{m}\}= \mathbb{Z}_{m}.
$$
The notion of permutative cellular automata was first introduced by Hedlund \cite{Hed-MST1969}. A linear local rule $f$ is permutative at the index $j$ if and only if $\mathrm{gcd}(\lambda_j,m)=1$, where $\mathrm{gcd}(p, q)$ denotes the greatest common divisor of $p$ and $q$.

For every linear local rule $f(x_l, \ldots, x_r) = \Sigma_{i=l}^r \lambda_i x_i \pmod{m}$, there associates a formal power series $F(X) = \Sigma_{i=l}^r \lambda_i X^{-i}$. Let $T_f$ be the cellular automaton defined by the local rule $f$. Ito et al. characterize the necessary and sufficient condition for the invertibility of $T_f$.

\begin{theorem}[See \cite{ION-JCSS1983}]
$T_f$ is invertible if and only if for each prime factor $p | m$ there exists a unique $j_p$ such that $f$ is permutative at the index $j_p$.
\end{theorem}

For the case where $m = p^k$ for some prime number $p$ and $k \in \mathbb{N}$, it comes immediately that $T_f$ is invertible if and only if there exists a unique $j$ such that $\mathrm{gcd}(\lambda_j, p) = 1$. Manzini and Margara demonstrate the corresponding formal power series $F(X)$ is invertible in $\mathbb{Z}_m \llbracket X, X^{-1} \rrbracket$.

\begin{theorem}[See \cite{MM-JCSS1998}] \label{thm:formula-InvF}
Suppose $m = p^k$ and $T_f$ is invertible. Write $F(X) = \lambda_{j_p} X^{-j_p} + p H(X)$. Let $\widetilde{H}(X) = - \lambda_{j_p} X^{j_p} H(X)$. Then
$$
F^{-1}(X) = \lambda_{j_p}^{-1} X^{j_p} ( 1 + p \widetilde{H}(X) + \cdots + p^{k-1} \widetilde{H}^{k-1}(X)),
$$
where $\lambda_{j_p}^{-1}$ is the inverse element of $\lambda_{j_p}$ in $\mathbb{Z}_m$.
\end{theorem}

%\begin{remark}
%An alternative viewpoint of the above result is to express $F(X) = \lambda_{j_p} X^{-j_p} (1 + p H(X))$. Then
%$$
%F^{-1}(X) = \lambda_{j_p}^{-1} X^{j_p} (1 - p H(X) + \cdots + (-1)^{k-1} p^{k-1} H^{k-1}(X)).
%$$
%\end{remark}

Let $T: X \rightarrow X$ be a measure-preserving transformation on a probability space $(X, \mathcal{B}, \mu)$. $T$ is called \emph{strong mixing} if
$$
\lim_{n \to \infty} \mu(T^{-n}A \cap B)=\mu(A) \mu(B)
$$
for any $A, B\in \mathcal{B}$. Furthermore, $T$ is called $k$-mixing if for every given $\{A_i\}_{i=0}^k \subset \mathcal{B}$,
$$
\lim_{n_1, n_2, \ldots, n_k \to \infty} \mu(A_0 \cap T^{-n_1} A_1 \cap \cdots \cap T^{-(n_1 + \cdots + n_k)} A_k)=\mu(A_0) \mu(A_1) \cdots \mu(A_k).
$$
It is seen that strong mixing is $1$-mixing.

It is known that every surjective cellular automaton preserves the uniform Bernoulli measure (cf.~\cite{CP-MST1974,Kle-PAMS1997,She-MM1992} for instance). For the rest of this paper, $\mu$ refers to the uniform Bernoulli measure unless stated otherwise. Kleveland \cite{Kle-PAMS1997} and Shereshevsky \cite{She-MM1992, She-IMN1997} have proved that $T_f$ is strong mixing if $r < 0$ (resp.~$l > 0$) and $f$ is left permutative (resp.~right permutative); some of these cellular automata are even $k$-mixing. Recently, one-sided expansive invertible cellular automata and two-sided expansive permutation cellular automata have been demonstrated to be strong mixing (see \cite{BM-IJM1997, BM-JMSJ2000, JNY-AAM2013, JY-AAM2007, Nasu-TAMS2002}). Theorem \ref{thm:ILCA-Mixing} addresses the necessary and sufficient condition for whether an invertible linear cellular automaton is strong mixing, which is an extension and characterizes the strong mixing property of invertible linear cellular automata completely.

\begin{theorem}\label{thm:ILCA-Mixing}
An invertible linear cellular automaton $T_f$ is strong mixing with respect to the uniform Bernoulli measure if and only if $j_p \neq 0$ for every prime factor $p$ of $m$.
\end{theorem}

The following corollary, which can be derived from the demonstration of Theorem \ref{thm:ILCA-Mixing} with a minor modification, addresses the ``opposite" result to the above theorem. Notably, a one-dimensional linear cellular automaton with local rule $f(x_{-r}, \ldots, x_r) = \Sigma_{i=-1}^r \lambda_i x_i \pmod{m}$ is ergodic if and only if $\mathrm{gcd}(\lambda_{-r}, \ldots, \lambda_{-1}, \lambda_1, \ldots, \lambda_r, m) = 1$ \cite{CFMM-TCS2000}. (In fact, a necessary and sufficient condition for the ergodicity of a multidimensional linear cellular automaton is demonstrated in \cite{CFMM-TCS2000}.) Corollary \ref{cor:ILCA-NotErgodic} indicates a concise if-and-only-if condition for the ergodicity of one-dimensional  invertible linear cellular automaton.

\begin{corollary}\label{cor:ILCA-NotErgodic}
An invertible linear cellular automaton $T_f$ is non-ergodic with respect to the uniform Bernoulli measure if and only if $j_p = 0$ for some prime factor $p$ of $m$.
\end{corollary}

Recall that $T$ is ergodic if and only if $\lim_{n \to \infty} \mu(T^{-n}A \cap B) > 0$ for any $A, B\in \mathcal{B}$ with positive measures. Examples \ref{eg:m=4} and \ref{eg:m=12} interpret an intuitive idea for the reliability of Theorem \ref{thm:ILCA-Mixing}; the rigorous proof is postponed to Section \ref{sec:Proof-mixing}.

A stronger property in ergodic theory is \emph{Bernoulli automorphism}. Given $\epsilon \geq 0$, two partitions $\xi =\{C_i\}$ and $\eta=\{D_j\}$ of the measure space $(\mathbb{Z}^{\mathbb{Z}}_m, \mathcal{B}, \mu)$ are said to be $\epsilon$-\emph{independent} if
$$
\overset{}{\underset{i,\ j}{\sum}}|\mu(C_i \cap D_j)-\mu(C_i)\mu(D_j)| \leq \epsilon.
$$
$\xi$ and $\eta$ are independent if $\epsilon = 0$. A partition $\xi =\{C_i\}$ is called \emph{Bernoulli} for an automorphism $T_f$ if there exists an integer $N>0$ such that the partitions $\overset{0}{\underset{k=-n}{\bigvee}}T^k\xi$ and $\overset{N+n}{\underset{k=N}{\bigvee}}T^k\xi$ are independent for all $n\geq 0$. Furthermore, $T_f$ is a Bernoulli automorphism if $T_f$ has a generating Bernoulli partition.

Suppose the local rule $f$ is permutative in the variable $x_r$ (resp.~$x_l$) and $0 \leq l < r$ (resp.~$l < r \leq 0$), Shereshevsky showed that the natural extension of the dynamical system $(\mathbb{Z}_m^{\mathbb{Z}}, \mathcal{B}, \mu, T_f)$ is a Bernoulli automorphism \cite{She-MM1992, She-IMN1997}. It is well-known that a Bernoulli automorphism is also strong mixing. With the similar discussion of the demonstration of Theorem \ref{thm:ILCA-Mixing}, we extend the results to the class of invertible linear cellular automata.

\begin{theorem}\label{thm:ILCA-Bernoulli}
An invertible linear cellular automaton $T_f$ is a Bernoulli automorphism with respect to the uniform Bernoulli measure if and only if $j_p \neq 0$ for every prime factor $p$ of $m$.
\end{theorem}

% -------------------------------------------------------------
\section{Examples} \label{sec:example}

This section clarifies the key ideas of the proof of Theorems \ref{thm:ILCA-Mixing} and \ref{thm:ILCA-Bernoulli} and Corollary \ref{cor:ILCA-NotErgodic} by examining three examples.

\begin{example}\label{eg:m=4}
Consider $m = 4 = 2^2$ and $f(x_1, x_2, x_3) = 2 x_1 + x_2 + 2 x_3 \pmod{4}$. It follows that $f^{-1}(x_{-3}, x_{-2}, x_{-1}) = 2 x_{-3} + x_{-2} + 2 x_{-1} \pmod{4}$, and
$$
f^{2n}(x_{4n}) = x_{4n} \pmod{4}, \qquad f^{-2n}(x_{-4n}) = x_{-4n} \pmod{4}.
$$
Given any other cylinder $[a_{L'}, \ldots, a_{R'}]_{L'}^{R'}$, there exists $N \in \mathbb{N}$ such that $R' < L+4n$ for $n \geq N$. It is well-known that all surjective CA preserve the uniform Bernoulli measure $\mu$. Therefore,
\begin{align*}
&\mu(T_f^{-k} [a_L, \ldots, a_R]_L^R \cap [a_{L'}, \ldots, a_{R'}]_{L'}^{R'}) \\
  &\qquad = \mu(T_f^{-k} [a_L, \ldots, a_R]_L^R) \mu([a_{L'}, \ldots, a_{R'}]_{L'}^{R'}) \\
  &\qquad = \mu([a_L, \ldots, a_R]_L^R) \mu([a_{L'}, \ldots, a_{R'}]_{L'}^{R'})
\end{align*}
for $k \geq 2N$. This demonstrates $T_f$ is strong mixing.
\end{example}

Next example investigates the case where $m$ has two distinct prime factors. The discussion can be extended to more general cases.

\begin{example}\label{eg:m=12}
Suppose $m = 12 = 2^2 \cdot 3$ and $f(x_0, x_1, x_2) = 6 x_0 + 3 x_1 + 2 x_2 \pmod{12}$. Let $\phi_4$ and $\phi_3$ be the canonical projections from $\mathbb{Z}_{12}$ to $\mathbb{Z}_4$ and $\mathbb{Z}_3$, respectively. Then $\Phi := \Phi_4 \times \Phi_3$ is an isomorphism from $\mathbb{Z}_{12}^{\mathbb{Z}}$ to $\mathbb{Z}_4^{\mathbb{Z}} \times \mathbb{Z}_3^{\mathbb{Z}}$, where $\Phi_4: \mathbb{Z}_{12}^{\mathbb{Z}} \to \mathbb{Z}_4^{\mathbb{Z}}$ and $\Phi_3: \mathbb{Z}_{12}^{\mathbb{Z}} \to \mathbb{Z}_3^{\mathbb{Z}}$ are obtained from $\phi_4$ and $\phi_3$, respectively.

Let $f_4$ and $f_3$ be defined as
\begin{align*}
f_4 (x_0, x_1, x_2) &= f(x_0, x_1, x_2) \pmod{4} \\
\intertext{and}
 f_3(x_0, x_1, x_2) &= f(x_0, x_1, x_2) \pmod{3},
\end{align*}
respectively. In other words,
\begin{align*}
f_4 (x_0, x_1, x_2) &= 2 x_0 + 3x_1 + 2x_2 \pmod{4}, \\
f_3(x_0, x_1, x_2) &= 2 x_2 \pmod{3}.
\end{align*}
Then the projections of $T_f$ on $\mathbb{Z}_4^{\mathbb{Z}}$ and $\mathbb{Z}_3^{\mathbb{Z}}$, denoted by $T_4$ and $T_3$, are the cellular automata with local rules $f_4$ and $f_3$, respectively. Furthermore, let $\mu_4$ and $\mu_3$ be the push-forward measures of $\Phi_4$ and $\Phi_3$, respectively. $\mu$ is the uniform Bernoulli measure on $\mathbb{Z}_{12}^{\mathbb{Z}}$ indicates that
\begin{enumerate}[\bf 1)]
\item $\mu_4$ and $\mu_3$ are the uniform Bernoulli measures on $\mathbb{Z}_4^{\mathbb{Z}}$ and $\mathbb{Z}_3^{\mathbb{Z}}$, respectively.
\item $\mu \cong \mu_4 \times \mu_3$.
\item $T_f \cong T_4 \times T_3$, and the diagram
$$
\xymatrix{
 \mathbb{Z}_{12}^{\mathbb{Z}} \ar[rr]^{T_{f}} \ar[d]_\Phi && \mathbb{Z}_{12}^{\mathbb{Z}} \ar[d]^{\Phi} \\
 \mathbb{Z}_4^{\mathbb{Z}} \times \mathbb{Z}_3^{\mathbb{Z}} \ar[rr]_{T_4\times T_3} && \mathbb{Z}_4^{\mathbb{Z}} \times \mathbb{Z}_3^{\mathbb{Z}}
}
$$
is commutative.
\end{enumerate}

Similar to the discussion of Example \ref{eg:m=4}, the local rule of $T_4^{-1}$ is expressed as
$$
f_4^{-1}(x_{-2}, x_{-1}, x_0) = 2 x_{-2} + 3x_{-1} + 2x_0 \pmod{4},
$$
and
$$
f_4^{2n}(x_{2n}) = x_{2n} \pmod{4}, \qquad f_4^{-2n}(x_{-2n}) = x_{-2n} \pmod{4}
$$
%\begin{align}
%&f_4^{2n}(x_{2n}) = x_{2n} \pmod{4}, \qquad f_4^{-2n}(x_{-2n}) = x_{-2n} \pmod{4}, \\
%&f_4^{2n+1}(x_{2n}, x_{2n+1}, x_{2n+2}) = 2 x_{2n} + 3x_{2n+1} + 2x_{2n+2} \pmod{4}, \\
%&f_4^{-(2n+1)}(x_{-2n-2}, x_{-2n-1}, x_{-2n}) = 2 x_{-2n-2} + 3x_{-2n-1} + 2 x_{-2n} \pmod{4},
%\end{align}
%for $n \in \mathbb{N}$. Given a cylinder $U = [a_L, \ldots, a_R]_L^R$, it is seen that, similarly as demonstrated in Example \ref{eg:m=4}, $T^{-n} U$ is a cylinder of the form:
%$$
%T_4^{-2n} U = [a_L, \ldots, a_R]_{L+2n}^{R+2n},
%$$
%and
%$$
%T_4^{-(2n+1)} U = [a_{L_2}, a_{L_1}, \overline{a}_{L+1}, \ldots, \overline{a}_{R-1}, a_{R_1}, a_{R_2}]_{L + 2n}^{R + 2n+2},
%$$
%for $n \in \mathbb{N}$, where
%$$
%\overline{a}_i = 2 a_{i-1} + 3 a_i + 2 a_{i+1} \quad \text{for} \quad i = L+1, \ldots, R-1,
%$$
%$a_{L_1} = 3 a_L + 2 a_{L+1}, 2 + 3 a_L + 2 a_{L+1}$,
%$$
%a_{L_2} = \left\{
%            \begin{array}{ll}
%              2 a_L \text{ or } 2 a_L + 2, & a_{L_1} = 3 a_L + 2 a_{L+1}\hbox{;} \\
%              2 a_L + 1 \text{ or } 2 a_L + 3, & a_{L_1} = 2 + 3 a_L + 2 a_{L+1}\hbox{.}
%            \end{array}
%          \right.
%$$
%$a_{R_1} = 3 a_R + 2 a_{R-1}, 2 + 3 a_R + 2 a_{R-1}$, and
%$$
%a_{R_2} = \left\{
%            \begin{array}{ll}
%              2 a_R \text{ or } 2 a_R + 2, & a_{R_1} = 3 a_R + 2 a_{R-1}\hbox{;} \\
%              2 a_R + 1 \text{ or } 2 a_R + 3, & a_{R_1} = 2 + 3 a_R + 2 a_{R-1}\hbox{.}
%            \end{array}
%          \right.
%$$
%$\mu_4$ is the uniform Bernoulli measure on $\mathbb{Z}_4^{\mathbb{Z}}$ inferring that $\mu_4(T_4^{-n} U) = \mu_4(U) = 4^{-(R-L+1)}$ for all $n$.
Given any other cylinder $U'= [a_{L'}, \ldots, a_{R'}]_{L'}^{R'}$, pick $N = \lfloor \frac{|R'-L|}{2} \rfloor + 1$. It follows immediately
\begin{equation}\label{eq:m=12-T4-mix-condition}
\mu_4(T_4^{-k} U_4 \cap U_4') = \mu_4(T_4^{-k} U_4) \mu_4(U_4') = \mu_4(U_4) \mu_4(U_4') \quad \text{for} \quad k \geq 2N
\end{equation}
since $L + k > R'$, where $A_j := \Phi_j(A)$ for $A \in \mathcal{B}$.

Similarly, the local rule of $T_3^{-1}$ is $f_3^{-1}(x_{-2}) = 2 x_{-2} \pmod{3}$. This infers
$$
f_3^{-n} = \left\{
            \begin{array}{ll}
              2 x_{-2n}, & \hbox{$n$ is odd;} \\
              x_{-2n}, & \hbox{$n$ is even.}
            \end{array}
          \right.
$$
%and
%$$
%T_3^{-n} U_3 = \mbox{}_{L + 2n} [\overline{a}_{L}, \ldots, \overline{a}_{R}]_{R + 2n} \quad \text{with} \quad
%\overline{a}_i = \left\{
%            \begin{array}{ll}
%              2 a_i, & \hbox{$n$ is odd;} \\
%              a_i, & \hbox{$n$ is even.}
%            \end{array}
%          \right.
%$$
It is seen that $\mu_3(T_3^{-n} U_3) = \mu_3(U_3)$ for all $n$ and
\begin{equation}\label{eq:m=12-T3-mix-condition}
\mu_3(T_3^{-k} U_3 \cap U_3') = \mu_3(T_3^{-k} U_3) \mu_3(U_3') = \mu_3(U_3) \mu_3(U_3') \quad \text{for} \quad k \geq N.
\end{equation}

Combining \eqref{eq:m=12-T4-mix-condition} and \eqref{eq:m=12-T3-mix-condition} we have, for $k \geq 2 N$,
\begin{align*}
\mu(T_f^{-k}U \cap U') &= (\mu_4 \times \mu_3)[(T^{-k}U \cap U')_4 \times (T^{-k}U \cap U')_3] \\
  &= \mu_4((T^{-k}U \cap U')_4) \mu_3((T^{-k}U \cap U')_3) \\
  &= \mu_4(T_4^{-k}U_4 \cap U'_4) \mu_3(T_3^{-k}U_3 \cap U'_3) \\
  &= (\mu_4 \times \mu_3) (U_4 \times U_3) \cdot (\mu_4 \times \mu_3) (U_4' \times U_3') = \mu(U) \mu(U').
\end{align*}
The strong mixing property of $T_f$ then follows.
\end{example}

Next example addresses that $T_f$ is even non-ergodic if $j_p = 0$ for some $p|m$.

\begin{example}\label{eg:j=0-not-ergodic}
Let $m = 36 = 2^2 \cdot 3^2$ and let $f$ be given as
$$
f(x_{-1}, x_0, x_1) = 15 x_{-1} + 10 x_0 + 6 x_1 \pmod{36}.
$$
It is seen that $j_2 = -1$ and $j_3 = 0$. To deduce that $T_f$ is not ergodic, we firstly observe that
$$
f^{6k} = 9 x_{-6k} + 28 x_0 \pmod{36} \quad \text{and} \quad f^{-6k} = 28 x_0 + 9 x_{6k} \pmod{36}
$$
for all $k \in \mathbb{N}$. Suppose $U = [0]_0$. Then $\alpha_0$ is a multiple of $9$ for each $\alpha = (\alpha_i) \in T_f^{-6k}(U)$. Hence $T_f$ is not ergodic since $T_f^{-6k}(U) \cap [1]_0 = \varnothing$ for all $k$.
\end{example}

% -------------------------------------------------------------
\section{Proof of Theorem \ref{thm:ILCA-Mixing}} \label{sec:Proof-mixing}

This section is devoted to demonstrating an invertible linear cellular automaton is strong mixing if and only if $j_p \neq 0$ for all prime factor $p$ of the integer $m$. One can verify without difficulty that the existence of prime factor $p$ of $m$ such that $j_p = 0$ infers such a cellular automaton is not ergodic, thus it is not strong mixing. (An intuitive exploration is that $j_p = 0$ for some prime factor $p$ indicates $T_p^n$ is a trivial shift map for some $n \in \mathbb{N}$, where $T_p \equiv T \pmod{p^k}$ and $p^k |m, p^{k+1} \nmid m$.) It remains to show the ``if" part of Theorem \ref{thm:ILCA-Mixing}.

Let $\mathcal{L}$ be the collection of linear local rules and let
$$
\mathbb{Z}_m[X, X^{-1}] = \{\Sigma_{i=l}^{r} a_i x^i, l, r \in \mathbb{Z}\}
$$
be the set of bi-polynomials of variable $X$. Define a mapping $\chi: \mathcal{L} \to \mathbb{Z}_m[X, X^{-1}]$, which relates a linear local rule $f$ to a bi-polynomial $F(X)$, as
$$
\chi (f) = \chi (\sum_{i=l}^{r} \lambda_i x_i) = \sum_{i=l}^{r} \lambda_i X^{-i} := F(X).
$$
It is easily seen that $\chi$ is bijective. Moreover, let $\mathbb{Z}_m \llbracket X, X^{-1} \rrbracket$ denote the formal power series generated by $\{X, X^{-1}\}$ over $\mathbb{Z}_m$. Then $\widehat{\chi}: \mathbb{Z}_m^{\mathbb{Z}} \to \mathbb{Z}_m \llbracket x, x^{-1} \rrbracket$ defined by
$$
\widehat{\chi}(\mathbf{b})=\sum^\infty_{i=-\infty} b_i X^i, \quad \text{where} \quad \mathbf{b} = (b_i)_{i \in \mathbb{Z}} \in \mathbb{Z}_m^{\mathbb{Z}}
$$
is also a bijection. Observe that, for each $\mathbf{b} = (b_i) \in \mathbb{Z}_m^{\mathbb{Z}}$,
\begin{align*}
\widehat{\chi}(T_{f}(\mathbf{b})) &= \widehat{\chi}\left[(\sum^{r+i}_{n=l+i} \lambda_{n-i} b_n)_i\right] = \sum^{\infty}_{i=-\infty} \left(\sum^{r+i}_{n=l+i} \lambda_{n-i} b_n\right) X^i, \\
\intertext{and}
\mathbb{T}(\widehat{\chi}(\mathbf{b})) &= \mathbb{T}\left(\sum^{\infty}_{i=-\infty} b_i X^i\right) \\
  &= \left(\sum^{r}_{n=l} \lambda_n X^{-n}\right) \left(\sum^{\infty}_{i=-\infty} b_i X^i\right) = \sum^{\infty}_{i=-\infty} \left(\sum^{r+i}_{n=l+i} \lambda_{n-i} b_n\right) X^i,
\end{align*}
where $\mathbb{T}(\Theta(X)) := F(X) \Theta(X)$. This implements that the diagram
\begin{equation}\label{diagram}
\xymatrix{
\mathbb{Z}_m^{\mathbb{Z}} \ar[rr]^T \ar[d]_{\widehat{\chi}} && \mathbb{Z}_m^{\mathbb{Z}} \ar[d]^{\widehat{\chi}} \\
\mathbb{Z}_m \llbracket X, X^{-1} \rrbracket \ar[rr]_{\mathbb{T}} && \mathbb{Z}_m \llbracket X, X^{-1} \rrbracket}
\end{equation}
commutes. Moreover, it follows immediately from the mathematical induction that $f^n = \chi^{-1} ((F(X))^n)$ for all $n \in \mathbb{N}$, where $f^n = f \circ f^{n-1}$.

Suppose $m = p^k$ for some prime number $p$ and $k \in \mathbb{N}$. Write $F(X)$ as $F(X) = \lambda_{j_p} X^{-j_p} + p H(X)$.
\begin{lemma}\label{lem:F^p-rule}
$(F(X))^{p^{k-1}} \equiv \lambda_{j_p}^m X^{-p^{(k-1)}j_p} \pmod{p^k}$.
\end{lemma}
\begin{proof}
Observe that
\begin{align*}
(\lambda_{j_p} X^{-j_p} + p H(X))^p &= \sum^p_{j=0} \binom{p}{j} (\lambda_{j_p} X^{-j_p})^j(p H(X))^{p-j} \\
  &\equiv \sum^p_{j=p-1} \binom{p}{j} (\lambda_{j_p} X^{-j_p})^j (p H(X))^{p-j} \pmod{p^2} \\
  &= p(\lambda_{j_p} X^{-j_p})^{p-1} (p H(X)) + (\lambda_{j_p} X^{-j_p})^p \\
  &\equiv \lambda_{j_p}^p X^{-p j_p} \pmod{p^2}
\end{align*}
Given $n \in \mathbb{N}$, let
$$
\phi(n) = \# \{i: 1 \leq i \leq n, \mathrm{gcd}(i, n) = 1\}
$$
be \emph{Euler's totient function}. Euler indicated that
$$
a^{\phi(n)} \equiv 1 \pmod{n} \quad \text{for all} \quad \mathrm{gcd}(a, n) = 1.
$$
More specifically, $a^{p^{k-1}} \equiv a^{p^k} \pmod{p^k}$ for all $\mathrm{gcd}(a, p) = 1$. This implies
$$
(\lambda_{j_p} X^{-j_p} + p H(X))^p \equiv \lambda_{j_p}^{p^2} X^{-p j_p} \pmod{p^2}
$$
Assume that Lemma \ref{lem:F^p-rule} holds for $m = p^{k-1}$, that is,
$$
(F(X))^{p^{k-2}} = p^{k-1} Q(X) + \lambda_{j_p}^{p^{k-2}} X^{-p^{(k-2)}j_p}
$$
for some $Q(X)$. Therefore,
\begin{align*}
(F(X))^{p^{k-1}} &= \left(p^{k-1} Q(X) + \lambda_{j_p}^{p^{k-2}} X^{-p^{(k-2)}j_p}\right)^p \\
  &= \sum^p_{j=0}\binom{p}{j}(p^{k-1} Q(X))^j\left(\lambda_{j_p}^{p^{k-2}} X^{-p^{(k-2)}j_p}\right)^{p-j} \\
  &\equiv \sum^1_{j=0}\binom{p}{j}(p^{k-1} Q(X))^j\left(\lambda_{j_p}^{p^{k-2}} X^{-p^{(k-2)}j_p}\right)^{p-j} \pmod{p^{k}} \\
  &\equiv \lambda_{j_p}^{p^{k-1}} X^{-p^{(k-1)}j_p} \quad \pmod{p^{k}} \\
  &\equiv \lambda_{j_p}^{p^{k}} X^{-p^{(k-1)}j_p} \pmod{p^{k}}
\end{align*}
This completes the proof.
\end{proof}

Suppose $T_f$ is an invertible linear cellular automaton. Theorem \ref{thm:formula-InvF} infers that $T_f^{-1}$ is a linear cellular automaton with local rule $f^{-1} = \chi^{-1}(F^{-1}(X))$, where
$$
F^{-1}(X) = \lambda_{j_p}^{-1} X^{j_p} ( 1 + p \widetilde{H}(X) + \cdots + p^{k-1} \widetilde{H}^{k-1}(X))
$$
and $\widetilde{H}(X) = - \lambda_{j_p} X^{j_p} H(X)$. Since $F^{-1}(X)$ is also of the form $F^{-1}(X) = \lambda_{j_p}^{-1} X^{j_p} + p \overline{H}(X)$, it follows from Lemma \ref{lem:F^p-rule} that
\begin{equation}
(F^{-1}(X))^{p^{k-1}} \equiv \lambda_{j_p}^{-m} X^{p^{k-1} j_p} \pmod{p^k}.
\end{equation}
For the clarification of the discussion, the notation $g \leftrightarrow [t_1, t_2]$ refers to the local rule $g(x_{t_1}, \ldots, x_{t_2}) = \Sigma_{i=t1}^{t2} \lambda_i x_i \pmod{m}$. Combining the one-to-one correspondence between $\mathcal{L}$ and $\mathbb{Z}_m[X, X^{-1}]$, Lemma \ref{lem:F^p-rule}, and the commutative diagram \eqref{diagram}, a straightforward examination deduces that
\begin{equation}\label{eq:f^n-range-multiple-of-m}
f^n \leftrightarrow [\ell p^{k-1} j_p], \quad f^{-n} \leftrightarrow [-\ell p^{k-1} j_p],
\end{equation}
if $n= \ell p^{k-1}$ for some $\ell \in \mathbb{N}$, and
\begin{align}
f^n &\leftrightarrow [\ell p^{k-1} j_p + \ell' l,\ \ell p^{k-1} j_p + \ell' r], \label{eq:f^n-range-non-multiple-of-m} \\
f^{-n} &\leftrightarrow [-\ell p^{k-1} j_p + \ell' \overline{l},\ -\ell p^{k-1} j_p + \ell' \overline{r}], \label{eq:f^-n-range-non-multiple-of-m}
\end{align}
if $n= \ell p^{k-1} + \ell'$ for some $\ell \in \mathbb{N}, 1 \leq \ell' < p^{k-1}$, where
$$
\overline{l} = (l-j_p)(k-1) - j_p, \quad \overline{r} = (r-j_p)(k-1) - j_p.
$$

The strong mixing property of invertible cellular automaton for the case where $m$ is a multiple power of a prime number follows via \eqref{eq:f^n-range-multiple-of-m}, \eqref{eq:f^n-range-non-multiple-of-m}, and \eqref{eq:f^-n-range-non-multiple-of-m}.

\begin{lemma}\label{lem:mixing-m=pk}
Suppose $m = p^k$ for some prime number $p$ and $k \in \mathbb{N}$. Then an invertible linear cellular automaton is strong mixing if and only if $j_p \neq 0$.
\end{lemma}
\begin{proof}
The ``only if" part follows immediately from \eqref{eq:f^n-range-multiple-of-m}. Given any two cylinders $U, V \in \mathbb{Z}_m^{\mathbb{Z}}$, write $U$ and $V$ as $U = [U]_{u_l}^{u_r}$ and $V = [V]_{v_l}^{v_r}$, respectively, for some $u_l, u_r, v_l, v_r \in \mathbb{Z}$. Herein the notation $[\alpha]_{i_1}^{i_2}$, $i_1 \leq i_2$, refers to the cylinder
\begin{align*}
\mbox{}_{i_1}[\alpha_{i_1}, \ldots, \alpha_{i_2}]_{i_2} = \{ x \in \mathbb{Z}_m^{\mathbb{Z}}: x_j = \alpha_j, i_1 \leq j \leq i_2\}.
\end{align*}

\begin{claim} \label{claim:TU-TinvU-cylinder}
$T_f^{-n} U$ is a finite union of cylinders for every $n \in \mathbb{Z}$.
\end{claim}
To see this, it suffices to show the case where $n = 1$.

Obviously $T_f^{-1} U \subseteq [U']_{u_l'}^{u_r'}$, where
\begin{align*}
u_l' &= \min \{u_l - (l-j_p)(k-1) + j_p, u_l - (r-j_p)(k-1) + j_p\}, \\
u_r' &= \max  \{u_r - (l-j_p)(k-1) + j_p, u_r - (r-j_p)(k-1) + j_p\}.
\end{align*}
$\mathrm{gcd}(\lambda_{j_p}^{-1}, m) \equiv 1 \pmod{m}$ indicates $f^{-1}(x_{l'}, \ldots, x_{r'})$ is a permutation at $x_{-j_p}$, where $l' = (l-j_p)(k-1) - j_p$, $r' = (r-j_p)(k-1) - j_p$, and $l' \leq -j_p \leq r'$ or $r' \leq - j_p \leq l'$. A straightforward and careful verification deduces
$$
(T_f^{-1}U)_i := \{x_i: x = (x_j) \in T_f^{-1} U\} = \mathbb{Z}_m
$$
provided
$$
i \in \mathbb{Z} \setminus \{\min\{u_l', u_r'\}, \min\{u_l', u_r'\} + 1, \ldots, \max\{u_l', u_r'\}\}.
$$
In other words, $T_f^{-1} U$ is a finite union of cylinders, and Claim \ref{claim:TU-TinvU-cylinder} follows.

Since $T_f^{-n} U$ is a finite union of cylinders, \eqref{eq:f^n-range-multiple-of-m} and \eqref{eq:f^-n-range-non-multiple-of-m} imply that there exists $N \in \mathbb{N}$ such that $T_f^{-n} U \subseteq [U']_{u_l'}^{u_r'}$ satisfies either $u_l' > v_r$ or $u_r' < v_l$ for $n \geq N$. It follows that
$$
\mu(T_f^{-n} U \cap V) = \mu(T_f^{-n} U) \mu(V) = \mu(U) \mu(V)
$$
for $n \geq N$ since $\mu$ is $T_f$-invariant. This demonstrates the strong mixing property of invertible linear cellular automata for the case where $m = p^k$.
\end{proof}

Notably, for every $n \in \mathbb{N}$ and cylinder $U$, neither $T_f^{-n} U$ nor $T_f^{n} U$ are cylinders in general. It is seen in the proof of Lemma \ref{lem:mixing-m=pk} that $T_f^{-1} U$ is a sub-cylinder of cylinder $[U']_{u_l'}^{u_r'}$ with $x_{u_l'}, \ldots, x_{u_r'}$ being constrained by some equations came from $f$ for all $x \in T_f^{-1} U$ (cf.~Examples \ref{eg:m=4} and \ref{eg:m=12}).

Suppose $m = sq$ for some coprime factors $s, q \in \mathbb{N}$. Define $f_s: \mathbb{Z}_s^{r-l+1} \to \mathbb{Z}_s$ and $f_q: \mathbb{Z}_q^{r-l+1} \to \mathbb{Z}_q$ by
\begin{align*}
f_s(x_{l}, \ldots, x_r) &= f(x_{l}, \ldots, x_r) \pmod{s} \\
\intertext{and}
f_q(x_{l}, \ldots, x_r) &= f(x_{l}, \ldots, x_r) \pmod{q},
\end{align*}
respectively. Then $f_s, f_q$ generate invertible cellular automata $T_s: \mathbb{Z}_s^{\mathbb{Z}} \to \mathbb{Z}_s^{\mathbb{Z}}$ and $T_q: \mathbb{Z}_q^{\mathbb{Z}} \to \mathbb{Z}_q^{\mathbb{Z}}$. Observe that the canonical isomorphism $\phi: \mathbb{Z}_m \to \mathbb{Z}_s \times \mathbb{Z}_q$ induces an isomorphism $\Phi: \mathbb{Z}_m^{\mathbb{Z}} \to \mathbb{Z}_s^{\mathbb{Z}} \times \mathbb{Z}_q^{\mathbb{Z}}$. A straightforward examination shows that the diagram
\begin{equation}\label{diagram:Zm-ZpxZq}
\xymatrix{
 \mathbb{Z}_m^{\mathbb{Z}} \ar[rr]^{T_{f}} \ar[d]_\Phi && \mathbb{Z}_m^{\mathbb{Z}} \ar[d]^{\Phi} \\
 \mathbb{Z}_s^{\mathbb{Z}} \times \mathbb{Z}_q^{\mathbb{Z}} \ar[rr]_{T_s\times T_q} && \mathbb{Z}_s^{\mathbb{Z}} \times \mathbb{Z}_q^{\mathbb{Z}}
}
\end{equation}
commutes.

Furthermore, let $\mu_s$ and $\mu_q$ be the push-forward measures of $\mu$ on $\mathbb{Z}_s^{\mathbb{Z}}$ and $\mathbb{Z}_q^{\mathbb{Z}}$ with respect to canonical projections $\Phi_s$ and $\Phi_q$, respectively. It follows that $\mu \cong \mu_p \times \mu_q$. For any two measurable sets $U, V \in \mathbb{Z}_m^{\mathbb{Z}}$ such that
\begin{align*}
\mu_s(\Phi_s(U) \cap \Phi_s(V)) &= \mu_s(\Phi_s(U)) \cdot \mu_s(\Phi_s(U)) \\
\intertext{and}
\mu_q(\Phi_q(U) \cap \Phi_q(V)) &= \mu_q(\Phi_q(U)) \cdot \mu_q(\Phi_q(V)),
\end{align*}
one can see that
\begin{align*}
(\mu_s \times \mu_q)(\Phi(U \cap V)) &= \mu_s(\Phi_s(U \cap V)) \cdot \mu_q(\Phi_q(U \cap V)) \\
  &= \mu_s(\Phi_s(U) \cap \Phi_s(V)) \cdot \mu_q(\Phi_q(U) \cap \Phi_q(V)) \\
  &= [\mu_s(\Phi_s(U)) \cdot \mu_s(\Phi_s(V))] \cdot [\mu_q(\Phi_q(U)) \cdot \mu_q(\Phi_q(V))] \\
  &= [\mu_s(\Phi_s(U)) \cdot \mu_q(\Phi_q(U))] \cdot [\mu_s(\Phi_s(V)) \cdot \mu_q(\Phi_q(V))] \\
  &= (\mu_s \times \mu_q) (\Phi(U)) \cdot (\mu_s \times \mu_q) (\Phi(V)) = \mu(U) \cdot \mu(V).
\end{align*}
In other words,
\begin{equation}\label{eq:mu-UV-is-muU-muV}
\mu(U \cap V) = \mu(U) \cdot \mu(V).
\end{equation}

For the general case, factorizing $m$ into the product of its prime factors $m = p_1^{k_1} p_2^{k_2} \cdots p_h^{k_h}$. Analogous discussion as above demonstrates that
\begin{enumerate}[\bf 1)]
\item The diagram
\begin{equation}
\xymatrix{
 \mathbb{Z}_m^{\mathbb{Z}} \ar[rr]^{T_{f}} \ar[d]_\Phi && \mathbb{Z}_m^{\mathbb{Z}} \ar[d]^{\Phi} \\
 \mathbb{Z}_{p_1^{k_1}}^{\mathbb{Z}} \times \cdots \times \mathbb{Z}_{p_h^{k_h}}^{\mathbb{Z}} \ar[rr]_{T_{p_1^{k_1}} \times \cdots \times T_{p_h^{k_h}}} && \mathbb{Z}_{p_1^{k_1}}^{\mathbb{Z}} \times \cdots \times \mathbb{Z}_{p_h^{k_h}}^{\mathbb{Z}}
}
\end{equation}
is commutative.
\item $\Phi := \Phi_{p_1^{k_1}} \times \cdots \times \Phi_{p_h^{k_h}}$ is an isomorphism, and $\mu \cong \mu_{p_1^{k_1}} \times \cdots \times \mu_{p_h^{k_h}}$.
\item For any two measurable sets $U, V \in \mathbb{Z}_m^{\mathbb{Z}}$ such that
$$
\mu_s(\Phi_s(U) \cap \Phi_s(V)) = \mu_s(\Phi_s(U)) \cdot \mu_s(\Phi_s(U)), \quad s = p_i^{k_i}, 1 \leq i \leq h,
$$
then
$$
\mu(U \cap V) = \mu(U) \cdot \mu(V).
$$
\end{enumerate}

In other words, we have demonstrated the following lemma.

\begin{lemma}\label{lem:T-mixing-iff-Ti-mixing}
Suppose $m = p_1^{k_1} p_2^{k_2} \cdots p_h^{k_h}$ for some prime number $p_1, \ldots, p_h$ and $k_1, \ldots, k_h \in \mathbb{N}$. A linear cellular automaton $T_f$ is strong mixing if and only if $T_{p_i^{k_i}}$ is strong mixing for $1 \leq i \leq h$.
\end{lemma}
\begin{proof}
Obviously, the strong mixing property of $T_f$ implies $T_{p_i^{k_i}}$ is strong mixing for $1 \leq i \leq h$. Suppose $T_{p_i^{k_i}}$ is strong mixing for $1 \leq i \leq h$. Given two cylinders $U$ and $V$, let $K_i$ be a positive integer, as indicated in the proof of Lemma \ref{lem:mixing-m=pk}, such that $\mu_{p_i^{k_i}}(T_{p_i^{k_i}}^{-n} U \cap V) = \mu_{p_i^{k_i}}(U) \mu_{p_i^{k_i}}(V)$ for $n \geq K_i$, where $1 \leq i \leq h$. Let $K = \max\{K_i\}$. A straightforward examination infers that
$$
\mu(T_f^{-n} U \cap V) = \mu(U) \mu(V) \quad \text{for} \quad n \geq K.
$$
This completes the proof.
\end{proof}

Notably Lemma \ref{lem:T-mixing-iff-Ti-mixing} remains true if one replaces strong mixing by either weak mixing or ergodic. The elucidation can be done via a minor modification of the discussion above and is omitted.

As a conclusion of this section, it is seen that Theorem \ref{thm:ILCA-Mixing} follows from Lemmas \ref{lem:mixing-m=pk} and \ref{lem:T-mixing-iff-Ti-mixing}.

% -------------------------------------------------------------
\section{Proof of Theorem \ref{thm:ILCA-Bernoulli}} \label{sec:Proof-Bernoulli}

This section focuses on the proof of Theorem \ref{thm:ILCA-Bernoulli}. Similar to the discussion in Section \ref{sec:Proof-mixing}, where the key ideas of the present elucidation are addressed in, it is not difficult to demonstrate that $T_f$ is not a Bernoulli automorphism if there exists a prime factor $p$ of $m$ such that $j_p = 0$. Hence it remains to show that $j_p \neq 0$ for all prime factors $p$ of $m$ implies $T_f$ is a Bernoulli automorphism.

Alternatively, an automorphism $T_f$ is Bernoulli if and only if there is a generator $\xi$ which is Bernoulli for $T_f$ \cite{Pet-1990}. Let $\xi_{n_1}^{n_2}$ denote the partition consists of all cylinders of the form $[U]_{n_1}^{n_2}$ for $n_1, n_2 \in \mathbb{Z}$. It follows immediately from \eqref{eq:f^n-range-multiple-of-m}, \eqref{eq:f^-n-range-non-multiple-of-m}, and the proof of Lemma \ref{lem:mixing-m=pk} that $\xi_{n_1}^{n_2}$ is a generator provided $n_2 - n_1 \geq r - l$.

Notably, Lemma \ref{lem:T-mixing-iff-Ti-mixing} infers that we may assume $m = p^k$ for some prime number $p$ and $k \in \mathbb{N}$ without loss of generality. Moreover, we assume that $r \geq l \geq 0$ for the clarification of the elucidation.

Set $\ell$ as the smallest positive integer satisfying $2 \ell \geq r - l$. Then $\xi_{-\ell}^{\ell}$ is a generator. Write $\xi_{-\ell}^{\ell} = \{U_i\}_{i=1}^{m^{2 \ell + 1}}$. Claim \ref{claim:TU-TinvU-cylinder}, which demonstrates that $T_f^{-n} U_i$ is a cylinder for $1 \leq i \leq m^{2 \ell + 1}$ and $n \in \mathbb{Z}$, together with equations \eqref{eq:f^n-range-multiple-of-m}, \eqref{eq:f^n-range-non-multiple-of-m}, and \eqref{eq:f^-n-range-non-multiple-of-m} shows that
\begin{equation}\label{eq:Tn-TNn-xi-n-multiple-pk}
\bigvee_{i=-n}^0 T_f^i \xi_{-\ell}^{\ell} \subseteq \xi_{-\ell}^{\ell + n j_p}, \quad
\bigvee_{i=N}^{N+n} T_f^i \xi_{-\ell}^{\ell} \subseteq \xi_{-\ell - (N + n) j_p}^{\ell - N j_p},
\end{equation}
if $n = c p^{k-1}$ for some $c \in \mathbb{N}$, and
\begin{align}
\bigvee_{i=-n}^0 T_f^i \xi_{-\ell}^{\ell} &\subseteq \xi^{\ell + (c p^{k-1} + 1) j_p + d (k-1)(r - j_p)}_{-\ell}, \label{eq:Tn-xi-n-not-multiple-pk} \\
\bigvee_{i=N}^{N+n} T_f^i \xi_{-\ell}^{\ell} &\subseteq \xi^{\ell - N j_p}_{-\ell - (c p^{k-1} + 1) j_p - d (k-1)(r - j_p)}, \label{eq:TNn-xi-n-not-multiple-pk}
\end{align}
if $n = c p^{k-1} + d$ for some $c \in \mathbb{N}, 1 \leq d < p^{k-1}$, herein
$$
N = t p^{k-1} \quad \text{and} \quad t = \max \left\{1, \left\lceil \dfrac{2 \ell}{p^{k-1} j_p} \right\rceil\right\}.
$$

Notably, both $\displaystyle \bigvee_{i=-n}^0 T_f^i \xi_{-\ell}^{\ell}$ and $\displaystyle \bigvee_{i=N}^{N+n} T_f^i \xi_{-\ell}^{\ell}$ are collection of cylinders of the form $[U]_{n_1}^{n_2}$ and $[V]_{n_1'}^{n_2'}$, respectively, where the indices $n_1, n_2, n_1'$, and $n_2'$, depend on the value of $n$, are addressed in \eqref{eq:Tn-TNn-xi-n-multiple-pk}, \eqref{eq:Tn-xi-n-not-multiple-pk}, \eqref{eq:TNn-xi-n-not-multiple-pk}. Analogous discussion as addressed in the proof of Lemma \ref{lem:mixing-m=pk} indicates that $\displaystyle \bigvee_{i=-n}^0 T_f^i \xi_{-\ell}^{\ell}$ and $\displaystyle \bigvee_{i=N}^{N+n} T_f^i \xi_{-\ell}^{\ell}$ are independent. Hence $T_f$ is an Bernoulli automorphism, and this completes the proof of Theorem \ref{thm:ILCA-Bernoulli}.

% -------------------------------------------------------------
\section{Conclusion and Discussion}\label{sec:discuss}

This paper investigates invertible linear cellular automata over $\mathbb{Z}_m^{\mathbb{Z}}$ with local rules of the form
$$
f(x_l, \ldots, x_r) = \Sigma_{i=l}^r \lambda_i x_i \pmod{m}, \quad l, r \in \mathbb{Z}, m \geq 2.
$$
Without using the natural extension, Theorems \ref{thm:ILCA-Mixing} and \ref{thm:ILCA-Bernoulli} reveal that an invertible linear cellular automaton is strong mixing and is a Bernoulli automorphism with respect to the uniform Bernoulli measure if and only if the canonical projection $f_p$ of $f$ is not permutative at the index $j = 0$ for every prime factor $p$ of $m$. This gives an affirmative answer for the open problem proposed by Pivato for reversible linear cellular automata \cite{Pivato-2009}. Furthermore, the elucidation extends the results in \cite{Kle-PAMS1997,She-MM1992} to all linear automorphisms. %We remark that the discussion in Section \ref{sec:Proof-mixing} works for demonstrating Theorem \ref{thm:ILCA-is-Kmixing}, which addresses the $M$-mixing property of invertible linear cellular automata, with a minor modification. This extends Shereshevsky's result in \cite{She-IMN1997}.

%\begin{theorem}\label{thm:ILCA-is-Kmixing}
%An invertible linear cellular automaton defined on $\mathbb{Z}_m^{\mathbb{Z}}$ is $M$-mixing for $M \in \mathbb{N}$ if and only if $j_p \neq 0$ for every prime factor $p$ of $m$.
%\end{theorem}

Notably, it can be verified without difficulty that an invertible linear cellular automaton is not ergodic if and only if $j_p = 0$ for some prime factor $p$ of $m$ (cf.~Corollary \ref{cor:ILCA-NotErgodic} and Example \ref{eg:j=0-not-ergodic}).

\begin{remark}\label{rmk:Result-for-more-measure}
Notably, one of the key points in demonstrating Theorem \ref{thm:ILCA-Mixing} is that the uniform Bernoulli measure $\mu$ is isomorphic to the product measure of those push-forward measures $\mu_{p_1^{k_1}} \times \cdots \times \mu_{p_h^{k_h}}$ under canonical projection maps. Hence Theorem \ref{thm:ILCA-Mixing} (resp.~Theorem \ref{thm:ILCA-Bernoulli}) remains true for every $T_f$-invariant measure $\mu$ which is isomorphic to the product measure $\mu_{p_1^{k_1}} \times \cdots \times \mu_{p_h^{k_h}}$ provided $T_{p_i^{k_i}}$ is strong mixing (resp.~Bernoulli) for $T_{p_i^{k_i}}$-invariant measure $\mu_{p_i^{k_i}}$ for $1 \leq i \leq h$.
\end{remark}

The methodology addressed in this paper can be applied to investigating multidimensional reversible linear cellular automata over $\mathbb{Z}_m$. Meanwhile, the elucidation of ergodic properties of nonlinear cases and cellular automata defined on Cayley graph are in preparation.

% -------------------------------------------------------------
%\section*{Acknowledgment}
%
%The author would like to thank Professor Huilan Chang for some interesting discussion during the preparation of this paper.

% bibliography --------------------------------------------------
\bibliographystyle{amsplain}
\bibliography{../../grece}

% -------------------------------------------------------------
\end{document}